\documentclass{amsart}

\usepackage[dvips]{graphics}

\usepackage{amssymb}
\usepackage{amsmath}

\usepackage{graphics}

\newtheorem{theorem}{Theorem}[section]

\newtheorem{corollary}[theorem]{Corollary}
\newtheorem{lemma}[theorem]{Lemma}

\theoremstyle{definition}
\newtheorem{definition}[theorem]{Definition}

\theoremstyle{remark}

\numberwithin{equation}{section}

\newcommand{\R}{\mathbb{R}}

\newcommand{\q}{\ensuremath{\pi}}
\newcommand{\imap}{\ensuremath{G}}

\newcommand{\F}{\ensuremath{\mathcal{F}}}

\newcommand{\dank}{\textsf{Acknowledgments.\ }}

\begin{document}


\title{Polar foliations and isoparametric maps}

\author{Marcos M. Alexandrino}


\address{Marcos M. Alexandrino \hfill\break\indent 
Instituto de Matem\'{a}tica e Estat\'{\i}stica\\
Universidade de S\~{a}o Paulo, \hfill\break\indent
 Rua do Mat\~{a}o 1010,05508 090 S\~{a}o Paulo, Brazil}
\email{marcosmalex@yahoo.de, malex@ime.usp.br}

\thanks{The first author was  supported by CNPq and partially supported by FAPESP.  }



\begin{abstract}
A singular Riemannian foliation $\F$ on a complete Riemannian manifold $M$ is called  a \emph{polar foliation}    
if, for each regular point $p$, there is an immersed submanifold $\Sigma$, called \emph{section},  that passes through $p$ and that meets all the leaves and always perpendicularly. A typical example of a polar foliation is the partition  of $M$ into the orbits  of a \emph{polar action}, i.e., an isometric action with sections.
In this work we 
prove that the leaves of $\F$ coincide with the 
level sets of a smooth map $H: M\to \Sigma$ if $M$ is simply connected. In particular, we have that the orbits of a polar action on a simply connected space are level sets of an isoparametric map.
This result extends previous results due to the author and Gorodski, Heintze, Liu and Olmos,  
Carter and West, and Terng.

\end{abstract}


\maketitle

\section{Introduction}

In this section, we  recall some definitions and state our main results as  Theorem \ref{theorem-submersionwithsingularities} and Corollary \ref{corollary-polaraction-isoparametricmap}.

\emph{A singular Riemannian foliation} $\F$  on a complete Riemannian manifold $M$ is a singular foliation  such that every geodesic  perpendicular  to one leaf   is perpendicular to every leaf it meets; see Molino~\cite{Molino}.
A leaf $L$ of  $\F$ (and each point in $L$) is called \emph{regular} if the dimension of $L$ is maximal, otherwise $L$ is called {\it singular}.

Typical examples of singular Riemannian foliation are the partition of a Riemannian manifold into the orbits of an isometric action, into the leaf closures of a Riemannian foliation and examples constructed by suspension of homomorphisms; see  \cite{Molino,Alex2,Alex4}.

A singular Riemannian foliation is called  a \emph{polar foliation} (or a singular Riemannian foliations with sections)   
if, for each regular point $p$, there is an immersed submanifold $\Sigma_{p}$, called \emph{section},  that passes through $p$ and that meets all the leaves and always perpendicularly. It follows that $\Sigma_p$ is totally geodesic and that the dimension of $\Sigma_{p}$ is equal to the codimension of the regular leaf $L_{p}$. This is equivalent to saying that  the normal distribution of the regular leaves of  $\F$  is integrable; see \cite{Alex4}. 
 
 Typical example of a polar foliation is the partition  of a Riemannian manifold into the orbits  of a \emph{polar action}, i.e., an isometric action with sections. Note that there are  examples of polar foliations on Euclidean spaces (i.e., \emph{isoparametric foliations}) with inhomogeneous leaves; see 
 Ferus \emph{et al.}~\cite{FerusKarcherMunzner}.  Others examples can be constructed by suspension of homomorphism, suitable changes of metric and surgery;
see \cite{Alex2,Alex4} and \cite{AlexToeben}.

We are now able to state our main result.

\begin{theorem}
\label{theorem-submersionwithsingularities}
Let $\F$ be a polar  foliation on a simply connected complete Riemannian manifold~$M$ and  $B=M/\mathcal F$. Let $\rho:M\rightarrow B$ denotes the canonical projection.   Then  $B$ is a good Coxeter orbifold and we can find a section $\Sigma$,  a canonical embedding $j:B\to \Sigma$ and a  smooth homeomorphism  $\imap: B\to B$  such that the map $H=\imap\circ \rho :M\to \Sigma$ is smooth.
  \end{theorem}

The above result generalizes a result due to the author and Gorodski \cite{AlexGorodski} for the case of polar foliations with flat sections
and previous results of Heintze \emph{et al}.~\cite{HOL}, Carter and West~\cite{CarterWest2},
and Terng~\cite{terng}
for isoparametric submanifolds. 
It can also be viewed as a converse to the main result 
in~\cite{Alex1} and in Wang~\cite{Wang1}.

For the particular case of polar actions, the above theorem can be reformulate as follows.

\begin{corollary}
\label{corollary-polaraction-isoparametricmap}
The orbits of a polar action on a simply connected complete Riemannian manifold $M$ are level sets of a isoparametric map $H:M\to \Sigma,$ where $\Sigma$ is a section.
\end{corollary}

Recall that a map $H:M\to \Sigma$ is called a \emph{transnormal map} if 
for each regular value $c\in \Sigma$ there exists a neighborhood  $V$ of $c$ in $\Sigma$  such that  for $U= H^{-1}(V)$ we have that $H \mid_{U}\to V$ is an integrable Riemannian submersion (with respect to some metric on  $V$). In addition, a transnormal map $H$ is said to be an \emph{isoparametric map}
if each regular level set has constant mean curvature. For  surveys on isoparametric maps, isoparametric submanifolds and their generalizations see Thorbergsson \cite{ThSurvey1,ThSurvey2,ThSurvey3}.

This paper is a \emph{preliminary version} and is organized as follows. In Section 2 we recall some facts about polar foliations and reflection groups. In Section 3 we prove Theorem \ref{proposition-reflectiongroups} that, together with the facts of Section 2, will imply Theorem \ref{theorem-submersionwithsingularities}.

\dank 
The author is very grateful to Dr. Alexander Lytchak  for very helpful discussions and in particular for suggesting  the question that motivates this work and for inspiring some ideas of the proof.

\section{Facts about polar foliations and reflection groups}

In this section we recall same facts about polar foliations and reflection groups that will be used in this work. The proofs of these facts can be found in \cite{Alex2,AlexToeben,AlexGorodski,Alex7} and in Davis~\cite{Davis}.

Several properties of polar actions are still true in the general context of polar foliations. One of them is the so called \emph{equifocality}(see \cite{Alex2,AlexToeben2}). 
This property implies that if  $\xi$ 
is a parallel normal  field  
along a curve $\beta:[0,1]\rightarrow L$ in a regular leaf  $L$,  
then the curve $t\mapsto \exp_{\beta(t)}(\xi(t))$ is  contained in the leaf 
$L_{\exp_{\beta(0)}(\xi)}$.

Now we  define a \emph{holonomy map} by setting  $\varphi_{[\beta]}: S_{\beta(0)}\rightarrow S_{\beta(1)}$ as 
$$\varphi_{[\beta]}(\exp(\xi(0)))=\exp(\xi(1)),$$ 
where $S_{\beta(i)}:=\{\exp_{\beta(i)}(\xi) | \xi\in \nu_{\beta(i)} L, \|\xi\|<\epsilon\}$ and  $[\beta]$ denotes the homotopy class of the curve $\beta$.
 The radius $\epsilon$ depends only on normal neighborhoods of the sections $\Sigma_{\beta(0)}$ and $\Sigma_{\beta(1)}$. Therefore the domain of the isometry 
$\varphi_{[\beta]}$  can contain singular points. Due to equifocality we have that $\varphi_{[\beta]}(x)\in L_{x}$. Hence, 
for a small $\epsilon,$ the map $\varphi_{[\beta]}$ coincides with the usual holonomy map.

We  define the  Weyl pseudogroup  as the pseudogroup generated by the local isometries $\varphi_{[\beta]}$ such that $\beta(0),\beta(1)\in \Sigma$. By the appropriate choice of section, this pseudogroup turns out to be a group that is called  \emph{Weyl group} $W(\Sigma)$; see \cite{Toeben}.
From now on we fix this section where the Weyl group is well defined as the section $\Sigma$.   

By definition, if $w\in W(\Sigma)$ then $w(x)\in L_{x}$ and $W(\Sigma)$ describes how the leaves of $\F$ intersect the section $\Sigma$. As expected, when $\F$ is  the partition of a Riemannian manifold $M$ into the orbits of a polar action $G\times M\to M$, the Weyl group $W(\Sigma)$ is the usual Weyl group $N/Z$, where $N=\{g\in G | \, g(x)\in \Sigma, \forall x\in \Sigma \}$ and $Z=\{g\in G | \, g(x)=x, \forall x\in \Sigma \}$; see \cite{PTlivro}.

Another result on polar actions that can be generalized to polar foliations is the fact that
the slice representations are polar. The result states that the infinitesimal foliation of $\F$ in the normal space of the leaves are polar, i.e.,  isoparametric, see \cite{Alex2}. This imply that, for each point $x,$ we can find a neighborhood $U$ such that the restriction of $\F$ to $U$ is diffeomorphic to an isoparametric foliation. 

In addition we have that the intersection of the singular leaves of $\F$ with the section $\Sigma$ is a union of totally geodesics hypersurfaces, called \emph{walls} and that the reflections in the walls are elements of the Weyl group $W(\Sigma)$; see \cite{Alex2}.
We  define  $\Gamma(\Sigma)$  as the group generated by these reflections.

The  fixed point set of a reflection $r\in\Gamma(\Sigma)$ can have more than one connected component if $\Sigma$ is not simply connected. In particular $\Gamma(\Sigma)$ is not a reflection group in the classical sense, as defined e.g., in Davis~\cite{Davis}. On the other hand,  recall that the fixed set points of a reflection is always a connected hypersurface if the space is simply connected.
This  motivates us to consider new groups acting on the Riemannian universal cover $\widetilde{\Sigma}$.

We start by considering the  section $\Sigma$. Let $w\in W(\Sigma)$. In what follows we will define a isometry $\tilde{w}:\widetilde{\Sigma}\to\widetilde{\Sigma}$ on the Riemannian universal cover of $\Sigma$ such that $\pi_{\Sigma}\circ\tilde{w}=w\circ\pi_{\Sigma}$ where $\pi_{\Sigma}:\widetilde{\Sigma}\to \Sigma$ is the Riemannian covering map.

\begin{definition}
Let $p$ be  regular point of the section $\Sigma$  and $\tilde{p}$ a point of $\widetilde{\Sigma}$ such that $\pi_{\Sigma}(\tilde{p})=p$.
Consider  a curve $c$ in $\Sigma$ that joins $p$ to $w(p)$ and $\delta$ be a curve in $\Sigma$ with $\delta(0)=p$. Let  $c *w\circ\delta$ be the concatenation  of $c$ and the curve $w\circ\delta$, $\widetilde{\delta}$ the lift of  $\delta$ starting at $\tilde{p}$ and $\widetilde{( c * w \circ \delta)}$ the lift of $c *w\circ\delta$ starting at $\tilde{p}$ .
Then we  set
\begin{equation}
\label{equation-new-def-tildew}
\tilde{w}_{[c]}(\widetilde{\delta}(1)):=\widetilde{( c * w \circ \delta)}(1).
\end{equation}
The isometry $\tilde{w}_{[c]}$ defined in equation \eqref{equation-new-def-tildew} is  called \emph{a lift of $w$ along $c$}. Clearly it depends on the homotopy class $[c]$.
\end{definition}

\begin{definition}
\label{def-lifted-Weyl-group}
We  call \emph{lifted Weyl group} $\widetilde{W}$ the group of isometies on the universal covering $\widetilde{\Sigma}$ generated by all isometries $\tilde{w}_{[c]}$ constructed above. In other words,
$\widetilde{W}=<\tilde{\omega}_{[c]}> $ for all $ w\in W$ and for all curves $ c:[0,1]\to \Sigma,$ such that $c(0)=p$ and $c(1)=w(p).$
\end{definition}

We are now able to construct the appropriate reflection group on $\widetilde{\Sigma}$.

\begin{definition}
Let $H$ be a wall in $\Sigma$ and $\widetilde{H}\subset \widetilde{\Sigma}$ an hypersurface such that $\pi_{\Sigma}(\widetilde{H})=H$.
Then a reflection $\tilde{r}$ in $\widetilde{H}$ is an element of $\widetilde{W}$. 
 The group generated by all reflections $\tilde{r}$ is called
\emph{lifted reflection group} $\widetilde{\Gamma}$.
\end{definition}

It follows from Davis  \cite[Lemma 1.1,Theorem 4.1]{Davis} that $\widetilde{\Gamma}$ acting on $\widetilde{\Sigma}$ is a reflection group in the classical sense, concept that we now recall.
Let $N$ be a simply connected Riemannian manifold. A \emph{reflection} $r:N \to N$ is an isometric involution such that the fixed point set $N_{r}$ separates $N$. It is possible to prove that $N_{r}$ is a connected totally geodesic hypersurface that separate $N$ in two components.
A  discrete subgroup $\Gamma$ of isometries of $N$ is called \emph{reflection group (in the classical sense) on $N$} if it is generated by reflections. 
Let $R$ be the set of reflections of $\Gamma$ and  $R(x)$ be the set of all $r\in R$ such that $x\in N_{r}.$  
A point $x\in N$ is \emph{nonsingular} if $R(x)=\emptyset$. Otherwise it is \emph{singular}. A \emph{chamber} of $\Gamma$ on $N$ is the closure of a connected component of the set of nonsingular points.
Let $C$ be a chamber. Denote by $V$ the set of reflections $v$ such that $R(x)=\{v\}$ for some $x\in C$. If $v\in V$ then $C_{v}=N_{v}\cap C$ is a \emph{panel} of $C$. Also $N_{v}$ is the \emph{wall supported by} $C_{v}$ and $V$ is the \emph{set of reflections through the panels of} $C$.
It turns out that $C$ is a fundamental domain of the action of $\Gamma$ on $N$,  the set $V$ generates $\Gamma$ and $(\widetilde{\Gamma},V)$ is a Coxeter system.

It follows from \cite{Alex7} that there exists a surjective homomorphism $\pi_{1}(M)\to \widetilde{W}/\widetilde{\Gamma}$. Therefore, if $M$ is simply connected, $\widetilde{W}=\widetilde{\Gamma}$ and hence
$M/\F=\Sigma/W(\Sigma)=\widetilde{\Sigma}/\widetilde{W}=\widetilde{\Sigma}/\widetilde{\Gamma}$.
As explained in \cite{Alex7}, the existence of such surjective homomorphism and the fact that $\widetilde{\Gamma}$ is a reflection group in the classical sense also allow us to give alternative proofs of the next two results, where the first is due to Lytchak \cite{Lytchak} and the second is due to the author and T\"{o}ben \cite{AlexToeben2}:
\begin{enumerate}
\item[(1)] If $M$ is simply connected, then the leaves are closed embedded.
\item[(2)] If $M$ is simply connected, the regular leaves have trivial holonomy and $M/\F$ is a simply connected Coxeter orbifold. 
\end{enumerate}

Let us now sum up the above discussion.

\begin{theorem}
\label{theorem-resumo-propriedades}
Let $\F$ be a polar foliation on a simply connected complete Riemannian manifold $M$. Then the leaves of $\F$ are closed embedded and the regular leaves have trivial holonomy. In addition $M/\F$ is a simply connected good Coxeter orbifold $N/\Gamma$ where $N$ is simply connected Riemannian manifold and $\Gamma$ is a reflection group. 
\end{theorem}

We conclude this section recalling the next result due the author and Gorodski \cite{AlexGorodski}.

\begin{theorem}
\label{teo-basic-forms}
Let $\mathcal F$ be a polar foliation on a complete Riemannian
manifold $M$, and let $\Sigma$ be a section 
of $\mathcal F$. Then the immersion of
$\Sigma$ into $M$ induces an isomorphism 
between the algebra of basic differential
forms on $M$ relative to $\mathcal F$
and the algebra of differential forms 
on $\Sigma$ which are invariant 
under the  Weyl pseudogroup of $\Sigma$.
\end{theorem}

\section{Maps and reflection groups}

In this section we will prove the next result.

\begin{theorem}
\label{proposition-reflectiongroups}
 Let $N$ be a simply connected complete Riemannian manifold. Let $\Gamma$ be a reflection
 group on $N$, let $C\subset N$ be a  chamber and let $\q:N\to C$ denotes the projection.  Then there is a map $\imap:C\to C$ with the following properties:
 \begin{enumerate}
 \item $\imap$ is a homeomorphism;
 \item The restriction of $\imap$ to each stratum is a diffeomorphism;
 \item The map $\imap$ is smooth;
 \item The composition $\imap\circ \q :N \to N$ is smooth.
 \end{enumerate}
\end{theorem}

The above result, Theorem \ref{theorem-resumo-propriedades} and Theorem \ref{teo-basic-forms}  will then imply  Theorem  \ref{theorem-submersionwithsingularities}.

\subsection{Motivation}
\label{subsection-motivation-G}
For the sake of motivation, in this subsection we prove the result in the particular case where $N=\R^2$ and $\Gamma$ is a finite group of 
reflections in lines $K_{1}=\cup \, \{y=a_{i}x\}$, following a procedure that will be generalized in the next section.

The strategy of the proof is the following.  Set $K_{0}=\{(0,0)\}$. First we will construct a $\Gamma$ invariant map $F_{1}:N \to N$ that will be smooth on $N-K_{0}$, whose restriction to each wall of each chamber will be the identity and so that $F_{1}\circ\q$ will  be smooth on $N-K_{0}$.
Then we will consider a smooth radial retraction $F_{0}$ (that is clearly  $\Gamma$-invariant) and the desired map will be $\imap=F_{0}\circ F_{1}$.  

In order to define the map $F_{1}$ we start by considering some real functions.

Let $h:[0,\infty)\to [0,\infty)$ be a smooth function such that
\begin{enumerate}
\item[(1)]  $h^{(n)}(0)=0$, for all $n$ (where $h^{(n)}$ is the $n$th derivative of $h$),
\item[(2)] $h(t)=t$ for $t\geq 1,$
\item[(3)]$h'(t)\neq 0$ for $t\neq 0$. 
\item[(4)] $h^{(n)}$ is a nonnegative increase function, for $t>0$ closer to $0$.
\end{enumerate}

Let $d_{1,0}:K_{1}\to\R$ be the distance function between  points of $K_{1}$ and $K_{0}$. Let $l_{1}:K_{1}\to\R$ be a smooth function such that
\begin{enumerate}
\item There are  some constant $A_{1,0}$ and $B_{1,0}$ such that 
\[ A_{1,0} d_{1,0}^{m}(p)\leq |l_{1}(p)|\leq B_{1,0} d_{1,0}^{n}(p).\]
 \item The image of $l_{1}$ is small enough so that $\{x+ v| v\in \nu_{x}K_{1}, \|v\|< l_{1}(x)\}$ is  a neighborhood $N_{1}$ of $K_{1}$. 
We also assume that $N_{1}$ is a union of disjoint neighborhoods $N_{1,k}$ such that $N_{1,k}$ is a neighborhood of a connected component $K_{1,k}$ of $K_{1}$ and $N_{1,k}\cap K_{1,j}=\emptyset$, for $k\neq j.$
\end{enumerate}

Now we  define $F_{1}$ on the neighborhood $N_{1}$ of $K_{1}$ as 
$$F_{1}(x+t v)=x+l_{1}(x)h(\frac{t}{l_{1}(x)})v,$$
where $x\in K_{1},$ $v$ is a unit normal vector to $K_{1},$ $0\leq t<l_{1}(x)$.
On $N-N_{1}$ the map $F_{1}$ is defined to be the identity. Since $h(t)=t$ for $t\geq 1$, this extension is continuous.
In this way we have constructed a $\Gamma$-equivariant continuous map $F_{1}:\R^{2}\to\R^{2}.$

In particular, if $C$ is a  chamber with a  wall $K_{1,0}:=\{y=a_{0}=0, x>0\},$ then the restriction of $F_{1}$ to a neighborhood $K_{1,0}$ is
$F_{1}(x,y)=(x,f_{1,2}(x,y)),$ where $f_{1,2}(x,y)=l_{1}(x)h(\frac{|y|}{l_{1}(x)})\frac{y}{|y|}$ for $y\neq 0$ and $f_{1,2}(x,0)=0.$
Here we can assume $l_{1}(x)=bx$, for  a small $b>0$.
Since $h^{(n)}(0)=0$ we have
\begin{equation}
\label{eq-derivadazero-exR2}
\|D^{m}f_{1,2}(x,0)\|=0.
\end{equation}

Equation \eqref{eq-derivadazero-exR2} implies that map $F_{1}$ is smooth on $\R^{2}-(0,0)$. 
Since $h'(t)\neq 0$ for $t\neq 0$, the map $F_{1}$ restricts to the regular stratum is a diffeomorphism.
Note that $F_{1}$ fails to be smooth at $(0,0)$. For example consider a sequence $\{x_{n}\}$ on $\{y=0,x>0\}$
that converges to $0$ and check that for two constant $k_1$ and $k_{2}$ such that $h'(k_1)\neq h'(k_2)$ we have $\frac{\partial f_{1,2}}{\partial y}(x_n,k_{1} l_{1}(x_{n}))\to h'(k_{1})$ and $\frac{\partial f_{1,2}}{\partial y}(x_n,k_{2} l_{1}(x_{n}))\to h'(k_{2}).$
   
Equation \eqref{eq-derivadazero-exR2} and the fact that the restriction of $\q$ to the regular stratum is an isometry imply
that  $F_{1}\circ\q$ is smooth on $\R^{2}-(0,0)$.

Also note that
\begin{equation}\label{eq-estimativa-li-exemplo}\|D^{n} F_{1}\|\leq \frac{A}{|l_{1}|^{n}}
\end{equation}
in the intersection of each compact set with $N_{1}$.

We define a map $F_{0}:\R^{2}\to\R^{2}$ as  $F_{0}(0,0)=(0,0)$ and 
$$F_{0}(x,y)=(h(\sqrt{x^{2}+y^{2}})\frac{x}{\sqrt{x^{2}+y^{2}}},h(\sqrt{x^{2}+y^{2}})\frac{y}{\sqrt{x^{2}+y^{2}}}) ,$$
for $(x,y)\neq (0,0)$. 
Clearly $F_{0}$ is a $\Gamma$-equivariant continuous map and a smooth map on $\R^{2}-(0,0)$. Since  $h'(t)\neq 0$ for $t\neq 0$ the map $F_{0}$ is a diffeomorphism on  $\R^{2}-(0,0).$ The properties of the function $h$  imply that $F_{0}$ is also smooth at $(0,0)$. 

Finally we consider $\imap=F_{0}\circ F_{1}$. From the above discussion we conclude that the maps $G$ and $G\circ\q$ are smooth on $R^{2}-(0,0)$ and that  $G$ is a $\Gamma$-equivariant continuous map whose restriction to each stratum is a diffeomorphism. 

Note that, in the  case where $C$ is a  chamber with a  wall $K_{1,0}:=\{y=a_{0}=0, x>0\}$ and $l_{1}(x)=bx$ we have  
\begin{equation}
\label{eq-lhospital}
 h^{(n)}(\sqrt{x^{2}+y^{2}})\leq h^{(n)}(\sqrt{x^{2}+b^{2}x^{2}}),
\end{equation}
for $(x,y)\in N_{1}$. This follows from the fact that $h^{(n)}$ is a nonnegative increase function, for $t>0$.

The properties of the function $h$ (e.g., $h^{(n)}(0)=0$ and  equation \eqref{eq-lhospital}), equation \eqref{eq-estimativa-li-exemplo}, and the fact that the restriction of $\q$ to the regular stratum is an isometry 
imply that the maps $G$ and $G\circ\q$  are also smooth at $(0,0).$ We invite the reader to check these straightforward calculations.


\subsection{Proof of Theorem \ref{proposition-reflectiongroups}}

The proof of the general case is based on some ideas presented in Subsection \ref{subsection-motivation-G}. 
First, we are going to define a new metric $\tilde{g}$  so that normal slices of the strata are flat and totally geodesics. 
Then we will define the desired map $G$ as compositions of $\Gamma$-invariant maps $F_i$, that are constructed using 
exponential map with respect to the metric $\tilde{g}$.
Finally, due the properties of the  metric $\tilde{g}$, we will be able to verify that these maps are smooth in a similar way as we have done in 
Subsection \ref{subsection-motivation-G}.  

Consider the stratification of $N$ by orbit types of the action of $\Gamma$ on $N$, denote $K_{0}$  the strata with the lower dimension and 
$K_{i}$ the union of strata with dimension $i+\dim K_{0}$. Then we have the stratification $\{K_i\}_{i=0,...,n}$ where $n\leq \dim N$.

\begin{lemma}
\label{lemma-metric-normal-flat}
There exist a $\Gamma$-invariant metric $\tilde{g}$ on $N$ and neighborhoods $N_i$ of each stratum $K_{i}$ with the following properties:
\begin{enumerate}
\item[(1)] For each $x\in K_{i}$, we find a slice $S_{x}^{i}:=\{\exp_{x}(\xi)| \xi\in\nu_{x}(K_{i}), \|\xi\|<\epsilon_{x}\}$ that is a flat totally geodesic submanifold contained in $N_{i}$.
\item[(2)] If $j>i$, $x\in K_{i}$ and $y\in K_{j}\cap S_{x}^{i},$ then $S_{y}^{j}\subset S_{x}^{i}$. 
\end{enumerate}
\end{lemma}
\begin{proof}

First note that the normal bundle of $K_{i}$ is trivial. 
This can be proved using the following facts:
\begin{itemize}
\item each stratum is totally geodesic;
\item the slice $S_{x}^{i}$ has dimension 1 or the intersection of $S_{x}^{i}$ with the singular stratification is diffeomorphic to a union of simplicial cones;
\item the interior of two different chambers can not be joined by a continuous curve without singular points.
\end{itemize}


The desired metric $\tilde{g}$ will be constructed by induction.

Using the fact that the normal bundle of $K_{0}$ is trivial, we can find a $\Gamma$-invariant metric metric $g_{0}$ on a neighborhood of $K_{0}$ such that $S_{x}^{0}$ is a flat totally geodesic submanifold. Considering the appropriate $\Gamma$-invariant partition of unity $h_{1}, h_{2}$, we set
$\tilde{g}_{0}=h_{1}g_{0}+h_{2}g.$ Clearly $\tilde{g}_{0}$ is a $\Gamma$-invariant metric and hence each stratum is totally geodesic with respect to $\tilde{g}_{0}$. This last fact and the fact that $S_{x}^{0}$ is totally geodesic (with respect to $\tilde{g}_{0}$) imply that $S_{y}^{1}\subset S_{x}^{0}$ for $y\in K_{1}\cap S_{x}^{0}$.

Now assume  that we have defined a $\Gamma$-invariant metric  $\tilde{g}_{i}$ with the following properties:
\begin{enumerate}
\item[(1)] For $j\leq i,$ $S_{x}^{j}$ is a flat totally geodesic submanifold,
\item[(2)] $S_{y}^{k}\subset S_{x}^{j}$ for  $y\in K_{k}\cap S_{x}^{j}$ and $j<k\leq 1+i$.
\end{enumerate}

Using the fact that the normal bundle of $K_{i+1}$ is trivial and the properties of the metric $\tilde{g}_{i}$, we can find a $\Gamma$-invariant metric  $g_{i+1}$ on a neighborhood $N_{i+1}$ of $K_{i+1}$ such that
\begin{enumerate}
\item[(3)] The slice $S_{x}^{i+1}$ with respect to $\tilde{g}_{i}$ is also a slice with respect to $g_{i+1}$.
\item[(4)] $S_{x}^{i+1}$ is a flat totally geodesic submanifold.
\item[(5)] $g_{i+1}$ coincides with $\tilde{g}_{i},$ in a neighborhood of $K_{i}\cup\ldots\cup K_{0}$.
\end{enumerate}

 Considering the appropriate $\Gamma$-invariant partition of unity $h_{1}, h_{2}$, we define a $\Gamma$-invariant metric 
$\tilde{g}_{i+1}=h_{1}g_{i+1}+h_{2}\tilde{g}_{i}.$ The fact that the strata are totally geodesic with respect to $\tilde{g}_{i+1}$ and properties $(1),\ldots, (5)$ above imply:

\begin{enumerate}
\item[(6)] For $j\leq i+1,$ $S_{x}^{j}$ is a flat totally geodesic submanifold,
\item[(7)] $S_{y}^{k}\subset S_{x}^{j}$ for   $y\in K_{k}\cap S_{x}^{j}$   and $j<k\leq 2+i$.
\end{enumerate}

Finally set $\tilde{g}:=\tilde{g}_{n-1}$.

\end{proof}

For  $i=0,..., n-1,$ we  are going to define maps $F_{i}:N\to N$   with the following properties:
\begin{enumerate}
\item $F_i$ is a  $\Gamma$-equivariant homeomorphism, whose restriction to each stratum is a diffeomorphism, 
\item the restriction of the map $F_i$ to $K_{n}\cup\cdots\cup K_{i}$ is smooth,
\item for $G_i = F_i \circ F_{i+1} \circ ...\circ F_{n-1},$ the maps $G_i$ and  $G_i \circ \q$ are smooth on $K_{n}\cup\cdots\cup K_{i}$. 
\end{enumerate}

Firstly, we need to  to consider some functions. 

Let $h:[0,\infty)\to [0,\infty)$ be a smooth function such that
\begin{enumerate}
\item  $h^{(n)}(0)=0$, for all $n$ (where $h^{(n)}$ is the $n$th derivative of $h$),
\item $h(t)=t$ for $t\geq 1,$
\item$h'(t)\neq 0$ for $t\neq 0$. 
\item $h^{(n)}$ is a nonnegative increase function, for $t>0$ closer to $0$.
\end{enumerate}

Let $d_{i,j}:K_{i}\cap N_{j}\to \mathbb{R}^{+}$ be the distance function (with respect to the metric $\tilde{g}$) between points of  $ K_{i}\cap N_{j}$ and the set $K_{j},$ where $j<i$ and $i>0$.

Finally, for $i\geq 0,$ let  $l_{i}:K_{i}\to\mathbb{R}$ be a smooth function such that
\begin{enumerate}
\item For $i>0$,   $p\in K_{i}\cap N_{j}$ (when $j<i$) and $p\notin N_{k}$ (when $j<k<i$) we have  
\[ A_{i,j} d_{i,j}^{m}(p)\leq |l_{i}(p)|\leq B_{i,j} d_{i,j}^{n}(p),\]
where $A_{i,j}$ $B_{i,j}$ are constants.
 \item the image of $l_{i}$ is small enough so that $\{v| v\in \nu_{x}K_{i}, \|v\|< l_{i}(x)\}$ is mapped diffeomorphically onto $N_{i}$
by  $\widetilde{\exp}^{\perp}$, i.e.,  the normal exponential map  with respect to the metric $\tilde{g}$. 
\end{enumerate}

Define a map $F_{i}$ on a neighborhood $N_{i}$ of $K_{i}$ as 
$$F_{i}(\exp(tv))=\widetilde{\exp}_{x}(l_{i}(x)h(\frac{t}{l_{i}(x)})v),$$ where $x\in K_{i},$ $v$ is a unit normal vector  to $K_{i}$ with footpoint $x$, $0\leq t<l_{i}(x)$ and 
$\widetilde{\exp}$ is  the  exponential map  with respect to the metric $\tilde{g}$. 
On $N-N_{i}$ the map $F_{i}$ is defined to be the identity.
 
 Clearly $F_i$ and $G_i = F_i \circ F_{i+1} \circ ...\circ F_{n-1}$  are  $\Gamma$-equivariant homeomorphisms and, by construction, $F_i$ restricted to each stratum is a diffeomorphism.
As we will see below, the maps $F_i$, $G_i$ and $G_{i}\circ\q$ also fulfill the other desired properties. Then we set $\imap:=G_{0}$.

\begin{lemma}
\label{lemma-conta-Rn}
\begin{enumerate}
\item The restriction of the map $F_i$ to $K_{n}\cup\cdots\cup K_{i}$ is smooth.
\item  The maps $G_i$ and  $G_i \circ \q$ are smooth on $K_{n}\cup\cdots\cup K_{i}$. 
\end{enumerate}
\end{lemma}
\begin{proof}
Due to Lemma \ref{lemma-metric-normal-flat}, we can reduce the proof to the Euclidean case.
Let $z$ be a  coordinate such that $(F_{n-1})_{z}$ is not the identity on a neighborhood of a point $p\in K_{n-1}.$ Then the properties of the function $h$ imply $\| D^{n}(F_{n-1})_{z}(p)\|=0.$  It follows from this equation and from the fact that $\pi|_{K_{n}}$ is an isometry  that $F_{n-1}$ and $F_{n-1}\circ\pi$ are smooth on $N_{n-1}$. By the chain rule we have  the next equation $$\| D^{n}F_{n-1}\|\leq \frac{A_{n-1}}{|l_{n-1}|^{n}},$$
 in the intersection of each compact subset with $N_{n-1}.$ 

Now assume  that the maps $F_{i}\circ F_{i+1}\circ\cdots\circ F_{n-1}$ and $F_{i}\circ F_{i+1}\circ\cdots\circ F_{n-1}\circ\pi$ are smooth on $N_{i}$. Also assume that  
 \begin{equation}
 \label{eq-estimativa-li}
 \|D^{n} (F_{i}\circ F_{i+1}\circ\cdots\circ F_{n-1})\|\leq \frac{A_{i}}{|l_{i}|^{n}}.
 \end{equation}
Let $z$ be a coordinate such that $(F_{i-1}\circ\cdots\circ F_{n-1})_{z}$ is not the identity on a neighborhood of a point $p\in K_{i-1}.$

Then the properties of $h$ and equation \eqref{eq-estimativa-li} imply
 $$\| D^{n}(F_{i-1}\circ\cdots\F_{n-1})_{z}(p)\|=0.$$ 
From this equation and from the fact that $\pi_{K_{n}}$ is an isometry we conclude that
$F_{i-1}\circ\cdots\circ F_{n-1}$ and $F_{i-1}\circ\cdots\circ F_{n-1}\circ\pi$ are smooth on $N_{i-1}$. 

We also have by the chain rule 
$$\|D^{n} (F_{i-1}\circ\cdots\circ F_{n-1})\|\leq \frac{A_{i-1}}{|l_{i-1}|^{n}}$$ in the intersection of  each compact set with $N_{i-1}$. 
The rest of the proof follows by induction.
\end{proof}


\bibliographystyle{amsplain}

\end{document}